\documentclass{birkjour}
 \newtheorem{thm}{Theorem}[section]
 \newtheorem{cor}[thm]{Corollary}
 \newtheorem{lem}[thm]{Lemma}
 
 \theoremstyle{definition}
 \newtheorem{defn}[thm]{Definition}
 \theoremstyle{remark}

 \numberwithin{equation}{section}
 
 \usepackage{mathtools,dsfont, amsthm, amsfonts, amssymb, amsmath}
 \usepackage{babel}

\begin{document}

%
%
%
%
%
%
%
%
%

\title[Vectorial eigenfunctions]
 {On the uniqueness of eigenfunctions for the vectorial $p$-Laplacian}

\author{Ryan Hynd}
\address{Department of Mathematics\br
University of Pennsylvania\br 
David Rittenhouse Laboratory\br 
209 South 33rd Street\br
Philadelphia, PA 19104-6395, USA}
\email{rhynd@math.upenn.edu}
\author{Bernd Kawohl}
\address{Department of Mathematics and Computer Science\br
Cologne University\br
Albertus-Magnus-Platz\br
D-50923 Cologne, Germany}
\email{kawohl@math.uni-koeln.de}
\author{Peter Lindqvist}
\address{Department of Mathematical Sciences\br
Norwegian University of Science and Techniology\br
N-7491 Trondheim, Norway}
\email{peter.lindqvist@ntnu.no}


\dedicatory{}

\begin{abstract}
We study a non-linear eigenvalue problem for vector-valued eigenfunctions and give a succinct uniqueness proof for minimizers of the associated Rayleigh quotient.
\end{abstract}

\maketitle
\section{Introduction}

While the nonlinear Dirichlet eigenvalue problem for the equation 
\begin{equation*}
\rm{div}(|\nabla u|^{p-2}\nabla u)+\lambda |u|^{p-2}u=0
\end{equation*}
has been thoroughly investigated for scalar functions $u:\Omega\mapsto \mathbb{R}$ on a bounded domain $\Omega$ in $\mathbb{R}^n$ and for $p\in (1,\infty)$, the corresponding vectorial problem for ${\bf u}=(u_1,u_2,\ldots, u_N)$ has gained considerably less attention.
The Dirichlet problem of minimizing the Rayleigh quotient 
\begin{equation}\label{I_N}
I_N({\bf v})=\frac{\displaystyle\int_\Omega(|\nabla v_1|^2 + |\nabla v_2|^2+\dots+|\nabla v_n|^2)^\frac{p}{2}dx}{\displaystyle\int_\Omega(|v_1|^2 + |v_2|^2+\dots+|v_n|^2)^\frac{p}{2}dx},\qquad 1<p<\infty,
\end{equation}
among all vector-valued functions ${\bf v}$ in the Sobolev space $W^{1,p}_0(\Omega;\mathbb{R}^N)$ leads to the Euler-Lagrange equations
\begin{equation}\label{ee}
\rm{div}(|D {\bf u}|^{p-2}\nabla u_k)+\Lambda |{\bf u}|^{p-2}u_k=0, \qquad k=1,2,\ldots,N.
\end{equation}
Here ${\bf u}\in W^{1,p}_0(\Omega;\mathbb{R}^N)$ is a minimizer, $|D{\bf u}|^2=|\nabla u_1|^2+\ldots+|\nabla u_N|^2$ and $|{\bf u}|^2=u_1^2+\dots+u_N^2$. The existence of a minimizer comes by the direct method in the calculus of variations and in this case we denote $I_N({\bf u})$ by $\Lambda_1$. In the linear case ($p=2$) the system (\ref{ee}) is decoupled into
\begin{equation*}
\Delta u_k+\Lambda u_k=0, \qquad k=1,2,\ldots, N. 
\end{equation*} 
Then everything reduces to the scalar case $N=1$ with Helmholtz's equation.

Not all solutions of \eqref{ee} are necessarily minimizers. $\Lambda_1$ is the smallest eigenvalue and there are higher eigenfunctions.

\begin{defn} We say that ${ \bf u}\in W^{1,p}_0(\Omega;\mathbb{R}^N)$ is an {\bf eigenfunction} if
\begin{equation*}
\int_\Omega |D {\bf u} |^{p-2}D {\bf u} \cdot D {\bf \varphi}\ dx= 
\Lambda \int_\Omega |{\bf u}|^{p-2}{\bf u} \cdot   {\bf \varphi} \ dx
\end{equation*}
for all testfunctions ${\bf \varphi}=(\varphi_1,\ldots,\varphi_N)\in C_0^\infty(\Omega; \mathbb{R}^N)$. Here $\Lambda$ is the correspon\-ding {\bf eigenvalue}.
\end{defn}
By elliptic regularity theory \cite{Li} it is known that ${\bf u} \in C(\Omega; \mathbb{R}^N)$. For irregular domains, however, ${\bf u}$ is not always in $C({\overline\Omega};\mathbb{R}^N)$. In any case the gradient is known to be well-defined in $\Omega$, because ${\bf u}\in C^{1,\alpha}_{\rm loc}(\Omega;\mathbb{R}^N)$ for a suitable $\alpha\in(0,1)$ depending on $p$. 

In what follows, we observe that if $\omega:\Omega\mapsto \mathbb{R}$ is an eigenfunction in the scalar case $N=1$, then for any constant vector ${\mathbf c}$
\begin{equation*} {\bf u} =(c_1\omega,\ldots , c_N\omega)={\bf c} \omega
\end{equation*}
is a vectorial eigenfunction. It turns out that the converse is true for the smallest eigenvalue. The following  result was proved by Brock and Man\'asevich in \cite{BM}, see also del Pino \cite{dP}. 

\begin{thm}\label{uniquenessA}
All minimizers ${\bf u}\in W^{1,p}_0(\Omega;\mathbb{R}^N)$ of $I_N$ are of the form ${\bf u}={\bf c}\omega$, where $\omega$ is the scalar minimizer.
\end{thm}

We shall present a streamlined proof of the theorem above using the identity of Lagrange (\ref{VecLagrange}). In the case $N=1$ the scalar minimizer is known to be unique, except that it can be multiplied by constants.  While \cite{BK} or \cite {KL} contain fairly simple proofs for the scalar case, the proof given here does not even require the use of Jensen's inequality.

The vectorial case in only one independent variable, in which $\Omega$ is just an open interval, was studied by M. Del Pino, who proved in \cite {dP} that all solutions 
${\bf u}={\bf u}(t)=(u_1(t),\dots, u_N(t))$ of the problem
\begin{equation}
\begin{cases}
\displaystyle \frac{d}{dt}\left(|{\bf u}'(t)|^{p-2}{\bf u}'(t)\right) + \Lambda |{\bf u}(t)|^{p-2}{\bf u}(t) = 0 \qquad{\rm in }\ (0,1), \\
\hspace{3.7truecm}{\bf u}(0)={\bf u}(1) = 0,
\end{cases}
\end{equation}
are just copies of the scalar ones, i.e. ${\bf u}={\bf c}\omega$. Thus the previous theorem is valid for {\bf all} eigenfunctions. Since his proof is not easily available, for the benefit of the reader we present it in Section 3 below, based on an excerpt of his thesis. 

Finally we study the vectorial fractional Rayleigh quotient
\begin{equation*}
J_N({\bf v})=\frac{\displaystyle\int_{\mathbb{R}^n}\!\!\int_{\mathbb{R}^n}\frac{|{\bf v}(y)-{\bf v}(x)|^p}{|y-x|^{n+sp}} \ dx\ dy}{\displaystyle\int_\Omega|{\bf v}(x)|^p\ dx}
\end{equation*}
for $s\in (0,1)$ and ${\bf v}\in W^{s,p}_0(\Omega;\mathbb{R}^N)$ with ${\bf v}=0$ in $\mathbb{R}^n\setminus \Omega$. The scalar case $N=1$ was studied in \cite{LL} and \cite{FP}. We shall show that again the vectorial minimizers are just copies of the scalar ones  in Theorem \ref{fracthm} below. It is worth mentioning that that the isolation of the first eigenvalue (i.e. the minimum of the Rayleigh quotient) is known only in the following cases: the scalar case $N=1$, the linear case $p=2$ and the o.d.e. case $n=1$. Establishing the \lq\lq spectral gap\rq\rq $ $ for the general case seems to be an open problem. 

\section{Vectorial Minimizers}
We will prove Theorem 1.2 by employing Lagrange's identity, which asserts
\begin{equation}\label{VecLagrange}
\left|\sum^N_{i=1}t_i{\bf v}_i\right|^2=\sum^N_{i=1}(t_i)^2\sum^N_{i=1}\left|{\bf v}_i\right|^2-\sum_{1\le i<j\le N}|t_i{\bf v}_j-t_j{\bf v}_i|^2
\end{equation}
for $t_1,\dots, t_N\in \mathbb{R}$ and ${\bf v}_1,\dots ,{\bf v}_N\in \mathbb{R}^d$.   This identity is usually stated for $d=1$. Nevertheless, it is valid for all dimensions $d\ge 1$.  

\begin{proof}[Proof of Theorem 1.2] 
Let ${\bf u}=(u_1,u_2,\ldots, u_N)$  be a minimizer of the Rayleigh quotient \eqref{I_N} on $W^{1,p}_0(\Omega;\mathbb{R}^N)$ and $\Lambda_1=I_N({\bf u})$. Set $w=|{\bf u}|$. It is routine to verify $w\in W_0^{1,p}(\Omega)$, and direct computation gives 
\begin{equation}\label{BerndIdentity}
w\nabla w=\sum^N_{i=1}u_i\nabla u_i
\end{equation}
on $\{w>0\}$. By Lagrange's identity \eqref{VecLagrange}, we also have
\begin{equation}\label{Lagrangew}
w^2|\nabla w|^2=w^2|D{\bf u}|^2-\sum_{1\le i<j\le N}|u_i\nabla u_j-u_j\nabla u_i|^2
\end{equation}
on $\{w>0\}$. As $\nabla w=0$ almost everywhere on $\{w=0\}$,  we deduce
\begin{equation}\label{DwInequality}
|\nabla w|\le |D{\bf u}|
\end{equation}
almost everywhere in $\Omega$.  Furthermore,  
\begin{equation}\label{EigenvalueInequality}
\lambda_1\le \frac{\displaystyle\int_{\Omega}|\nabla w|^pdx}{\displaystyle\int_{\Omega}|w|^pdx}\le \frac{\displaystyle\int_{\Omega}|D{\bf u} |^pdx}{\displaystyle\int_{\Omega}|{\bf u}|^pdx}= \Lambda_1.
\end{equation}
Here $\lambda_1=\inf I_1$ is the smallest scalar eigenvalue.

\par Now suppose $u\in W^{1,p}_0(\Omega, \mathbb{R})$ is a scalar first eigenfunction and ${\bf c}\in\mathbb{R}^N$ with ${\bf c}\neq 0$. Then  ${\bf v}={\bf c} u\in W^{1,p}_0(\Omega, \mathbb{R}^N)$ and
$$ 
\Lambda_1\le \frac{\displaystyle\int_{\Omega}|D{\bf v} |^pdx}{\displaystyle\int_{\Omega}|{\bf v}|^pdx}= \frac{\displaystyle |{\bf c}|^p\int_{\Omega}|\nabla u|^pdx}{\displaystyle|{\bf c}|^p\int_{\Omega}|u|^pdx}=\frac{\displaystyle\int_{\Omega}|\nabla u|^pdx}{\displaystyle\int_{\Omega}|u|^pdx}=\lambda_1.
$$
In view of \eqref{EigenvalueInequality}, $\lambda_1=\Lambda_1$, so $w$ is necessarily a scalar first eigenfunction and \emph{equality must hold almost everywhere in \eqref{DwInequality}}.  Harnack's inequality for the $p$-Laplacian (see \cite{T}) also gives $w>0$ in $\Omega$. Therefore, by \eqref{Lagrangew}
\begin{equation}\label{loguijIdentity}
u_i\nabla u_j=u_j\nabla u_i \quad\text{ almost everywhere in }\Omega
\end{equation}
for all $i,j=1,\dots, N$.  Combining \eqref{loguijIdentity} and \eqref{BerndIdentity} gives
$$
w^2\nabla u_i=\sum^N_{j=1}u_j(u_j\nabla u_i)= \sum^N_{j=1} u_j (u_i\nabla u_j)=u_i \sum^N_{j=1} u_j\nabla u_j=u_i(w \nabla w).
$$
 That is, 
$$
w\nabla u_i=u_i \nabla w
$$
almost everywhere for $i=1,\dots, N$.  Since 
$$
\nabla\left(\frac{u_i}{w}\right)=\frac{w\nabla u_i-u_i \nabla w}{w^2}=0
$$
almost everywhere and $\Omega$ is connected, $u_i=c_iw$ for some $c_i\in \mathbb{R}$ and $i=1,\dots, N$. We conclude
$$
{\bf u}=(c_1,\dots, c_N) w={\bf c} w.
$$
\end{proof}

\section{One independent variable}
In this section we treat the o.d.e. case $n=1$, based on an excerpt of Manuel del Pino's work \cite{dP} in an interval, say $(0,1)\subset\mathbb{R}$. The Euler-Lagrange equations \eqref{ee} for ${\bf u}(t)=(u_1(t),\ldots, u_N(t))$ reduce to ordinary differential equations
\begin{equation}\label{N1ODE}
-\left(|{\bf u}'|^{p-2}{\bf u}'\right)'=\Lambda|{\bf u}|^{p-2}{\bf u}.
\end{equation}
A smooth ${\bf u}:(0,\infty)\rightarrow \mathbb{R}^N$ which satisfies \eqref{N1ODE},
$$
{\bf u}(0)=0,\text{ and } {\bf u}'(0)={\bf c}
$$
is 
\begin{equation}\label{explicitVecSoln}
{\bf u}(t)=\frac{{\bf c}}{\Lambda^{1/p}} \omega\left(\Lambda^{1/p}t\right).
\end{equation}
Here $\omega:\mathbb{R}\rightarrow \mathbb{R}$ is the solution of
\begin{equation}\label{littleuode}
\begin{cases}
-\left(|\omega'|^{p-2}\omega'\right)'=|\omega|^{p-2}\omega, \quad\text{ in } (0,1)\\
\qquad \omega(0)=0\quad \omega'(0)=1.
\end{cases}
\end{equation}
Such a scalar function can be found through integration. Moreover, each scalar eigenfunction on $\Omega=(0,1)$ is given by a multiple of
$$
\omega\left(k\lambda_p^{1/p}t\right)
$$
for some $k\in \mathbb{N}$. The corresponding eigenvalue (see \cite{E}, \cite{O}) is $k^p\lambda_p$ with 
$$
\lambda_p=\frac{(2\pi)^p(p-1)}{[p\sin(\pi/p)]^p}.
$$

\par We claim the initial value problem associated with the o.d.e.  \eqref{N1ODE} admits a unique solution.  This will be essentially due to the 
following lemma. 

\begin{lem}
Suppose ${\bf c}\in \mathbb{R}^N$ with ${\bf c}\neq 0$ and $t_0\ge 0$. There is $\epsilon>0$ so that the initial value problem
\begin{equation}\label{Uode}
\begin{cases}
-\left(|{\bf u}'|^{p-2}{\bf u}'\right)'=|{\bf u}|^{p-2}{\bf u}, \quad \text{ in } (t_0,t_0+\epsilon)\\
\qquad{\bf u}(t_0)=0\quad {\bf u}'(t_0)={\bf c}
\end{cases}
\end{equation}
has a unique solution. 
\end{lem}
\begin{proof}
Without any loss of generality, we may suppose $t_0=0$. We note that the second-order
initial value problem \eqref{Uode} is equivalent to the following first order one
\begin{equation*}\label{UVode}
\begin{cases}
\quad{\bf u}'=|{\bf v}|^{q-2}{\bf v} \qquad\text{ in } (0,\epsilon),\\
\quad{\bf v}'=-|{\bf u}|^{p-2}{\bf u} \quad\,\text{ in } (0,\epsilon),\\
\quad{\bf u}(0)=0,\qquad {\bf v}(0)=|{\bf c}|^{p-2}{\bf c}
\end{cases}
\end{equation*}
for ${\bf u}$ and ${\bf v}$.  Here $q=p/(p-1)$.
Furthermore, if $p\ge 2$, then ${\bf u}\mapsto |{\bf u}|^{p-2}{\bf u}$ is continuously differentiable on $\mathbb{R}^N$, while ${\bf v}\mapsto |{\bf v}|^{q-2}{\bf v}$ is smooth in a neighborhood of each point in $\mathbb{R}^N$ aside from the origin. In particular, the mapping 
\begin{equation}\label{eFFmapp}
F({\bf u},{\bf v})=(|{\bf v}|^{q-2}{\bf v},-|{\bf u}|^{p-2}{\bf u})
\end{equation}
on $\mathbb{R}^N\times \mathbb{R}^N$ is continuously differentiable in a neighborhood of $({\bf u},{\bf v})=(0,|{\bf c}|^{p-2}{\bf c})$. The claim then follows from the Picard-Lindel\"of theorem.

\par Let us now consider the case $1<p<2$.  In view of \eqref{explicitVecSoln}, a solution ${\bf u}$ exists globally. We just need to show that it is unique. Suppose ${\bf u}$ and ${\bf z}$ are two solutions in question for some $\epsilon>0$. Then  
\begin{equation}\label{EasyUZid}\begin{cases}
\quad \displaystyle|{\bf v}'(t)|^{p-2}{\bf u}'(t)-|{\bf z}'(t)|^{p-2}{\bf z}'(t)\\
=\displaystyle\int^t_0\frac{d}{ds}\left(|{\bf u}'(s)|^{p-2}{\bf u}'(s)-|{\bf z}'(s)|^{p-2}{\bf z}'(s)\right)ds\\
=-\displaystyle\int^t_0\left(|{\bf u}(s)|^{p-2}{\bf u}(s)-|{\bf z}(s)|^{p-2}{\bf z}(s)\right)ds\quad\text{ for }t\in[0,\epsilon].
\end{cases}\end{equation}

\par Let us recall two basic inequalities 
\begin{align*}
||{\bf a}|^{p-2}{\bf a}-|{\bf b}|^{p-2}{\bf b}|&\le (3-p)|{\bf b}-{\bf a}|\int^1_0|{\bf a}+\tau({\bf b}-{\bf a})|^{p-2}d\tau\\
&\leq 2|{\bf b}-{\bf a}|\int^1_0|{\bf a}+\tau({\bf b}-{\bf a})|^{p-2}d\tau\\
\end{align*}
and 
\begin{equation*} 
||{\bf a}|^{p-2}{\bf a}-|{\bf b}|^{p-2}{\bf b}|\ge (p-1)|{\bf b}-{\bf a}|\left(1+|{\bf a}|^2+|{\bf b}|^2\right)^{\frac{p-2}{2}}\\
\end{equation*}
which hold for each ${\bf a},{\bf b}\in \mathbb{R}^N$ and $1<p\le 2$. Combining these inequalities with \eqref{EasyUZid} gives
\begin{align*}
(p-1)&|{\bf u}'(t)-{\bf z}'(t)|\left(1+|{\bf u}'(t)|^2+|{\bf z}'(t)|^2\right)^{\frac{p-2}{2}}\\
&\quad\le \left||{\bf u}'(t)|^{p-2}{\bf u}'(t)-|{\bf z}'(t)|^{p-2}{\bf z}'(t)\right|\\
&\quad\le \int^t_0\left||{\bf u}(s)|^{p-2}{\bf u}(s)-|{\bf z}(s)|^{p-2}{\bf z}(s)\right|ds\\
&\quad \le2 \max_{0\le s\le t}|{\bf u}(s)-{\bf z}(s)|\int^t_0\int^1_0\frac{1}{|{\bf u}(s)+\tau({\bf z}(s)-{\bf u}(s))|^{2-p}}d\tau ds\\
&\quad \le2 \max_{0\le s\le t}|{\bf u}(s)-{\bf z}(s)|\int^\epsilon_{0}\int^1_0\frac{1}{|{\bf u}(s)+\tau({\bf z}(s)-{\bf u}(s))|^{2-p}}d\tau ds\\
&\quad =2 \max_{0\le s\le t}|{\bf u}(s)-{\bf z}(s)|\int^1_0\int^\epsilon_{0}\frac{1}{|{\bf u}(s)+\tau({\bf z}(s)-{\bf u}(s))|^{2-p}}dsd\tau 
\end{align*}
for each $0\le t\le \epsilon$. 

\par As ${\bf u}(0)={\bf z}(0)={\bf c}\neq 0$, 
$$
{\bf u}(s)+\tau({\bf z}(s)-{\bf u}(s))= {\bf c}s+o(s)
$$
as $s\rightarrow 0$ uniformly in $\tau\in [0,1]$. Reducing $\epsilon>0$ if necessary, we may suppose
$$
|{\bf u}(s)+\tau({\bf z}(s)-{\bf u}(s))|\ge \frac{1}{2}|{\bf c}|s\quad\text{ for }s\in[0,\epsilon].
$$
It follows that 
$$
\int^\epsilon_{0}\frac{1}{|{\bf u}(s)+\tau({\bf z}(s)-{\bf u}(s))|^{2-p}}ds\le \left(\frac{1}{2}|{\bf c}|\right)^{p-2}\frac{\epsilon^{p-1}}{p-1}.
$$

\par Since ${\bf u}$ and ${\bf z}$ are continuously differentiable on $[0,\epsilon)$, we also have a constant $K$ such that 
$$
|{\bf u}'(t)|,|{\bf z}'(t)|\le K \quad\text{ for }t\in[0,\epsilon]. 
$$ 
Therefore, 
$$
\left(1+|{\bf u}'(t)|^2+|{\bf z}'(t)|^2\right)^{\frac{p-2}{2}}\ge (1+2K^2)^{\frac{p-2}{2}}.
$$
Putting these estimates together yields
$$
\max_{0\le t\le \epsilon}|{\bf u}'(t)-{\bf z}'(t)|\le\frac{2}{(p-1)^2}\frac{(|{\bf c}|/2)^{p-2}}{(1+2K^2)^{\frac{p-2}{2}}} \epsilon^{p-1}\cdot \max_{0\le t\le \epsilon}|{\bf u}(t)-{\bf z}(t)|$$
$$=C\epsilon^{p-1}\cdot \max_{0\le t\le \epsilon}|{\bf u}(t)-{\bf z}(t)|
$$
Furthermore, we can integrate the left hand side over $[0,\epsilon]$ to get 
$$
\max_{0\le t\le \epsilon}|{\bf u}(t)-{\bf z}(t)|\le C\epsilon^{p}\cdot \max_{0\le t\le \epsilon}|{\bf u}(t)-{\bf z}(t)|
$$
For $C\epsilon^p<1$ we get a contradiction unless $\max_{0\le t\le \epsilon}|{\bf u}(t)-{\bf u}(t)|=0$. 
\end{proof}
This leads to the following theorem. 
\begin{thm}
Suppose $\Lambda\neq 0$ and ${\bf a},{\bf b}\in \mathbb{R}^N$. 
The initial value problem
\begin{equation*}\label{Uodelong}
\begin{cases}
-\left(|{\bf u}'|^{p-2}{\bf u}'\right)'=\Lambda |{\bf u}|^{p-2}{\bf u} \quad \text{ in } (0,\infty),\\
\hskip1.46truecm{\bf u}(0)={\bf a},\qquad {\bf u}'(0)={\bf b}
\end{cases}
\end{equation*}
has a unique solution. 
\end{thm}
\begin{proof}
By scaling, we may suppose $\Lambda =1$.  This second order initial value problem is equivalent to the first order initial value problem for ${\bf u}$ and ${\bf v}$
\begin{equation*}
\begin{cases}
\quad{\bf u}'=|{\bf v}|^{q-2}{\bf v} \qquad\text{ in } (0,\infty),\\
\quad{\bf v}'=-|{\bf u}|^{p-2}{\bf u} \,\quad\text{ in } (0,\infty),\\
{\bf u}(0)={\bf a},\qquad {\bf v}(0)=|{\bf b}|^{p-2}{\bf b}.
\end{cases}
\end{equation*}
Here $q=p/(p-1)$. Moreover, ${\bf v}$ would satisfy
\begin{equation*}
\begin{cases}
-\left(|{\bf v}'|^{q-2}{\bf v}'\right)'= |{\bf v}|^{q-2}{\bf v} \quad\text{in } (0,\infty),\\
\hskip1.46truecm{\bf v}(0)=|{\bf b}|^{p-2}{\bf b},\qquad {\bf v}'(0)=-|{\bf a}|^{p-2}{\bf a}.
\end{cases}
\end{equation*}

\par Note that any solution ${\bf u}$ and ${\bf v}$ must also fulfill the identity 
$$
\frac{1}{p}|{\bf u}(t)|^p+\frac{1}{q}|{\bf v}(t)|^q=\frac{1}{p}|{\bf a}|^p+\frac{1}{q}|{\bf b}|^p.
$$
So if ${\bf a}={\bf b}=0$, then both ${\bf u}$ and ${\bf v}$ vanish identically and the unique solution in question is ${\bf u}={\bf v}\equiv 0$. Otherwise, suppose ${\bf a}$ or ${\bf b}$ is not zero. In this case both ${\bf u}(t)$ and ${\bf v}(t)$ cannot vanish simultaneously for any solution. If for a given $t_0$, ${\bf u}(t_0)\neq 0$ and ${\bf v}(t_0)\neq 0$, then we can uniquely continue this solution on $(t_0,t_0+\epsilon)$ for some $\epsilon>0$ as the mapping \eqref{eFFmapp} is Lipschitz in a neighborhood of $({\bf u}(t_0),{\bf v}(t_0))$. Alternatively, if ${\bf u}(t_0)= 0$ and ${\bf v}(t_0)\neq 0$, we can appeal to the lemma to 
uniquely continue the solution on $(t_0,t_0+\epsilon)$. If instead ${\bf u}(t_0)\neq 0$ and ${\bf v}(t_0)= 0$, we can apply the lemma to ${\bf v}$ as this function satisfies the same type of equation. In conclusion, at any time $t_0\ge 0$, we may uniquely continue the solution to a longer interval. It follows that a solution ${\bf u}$ is globally and uniquely defined. 
\end{proof}
Finally, we are able to conclude that each eigenfunction is a scalar one when $\Omega=(0,1)$. 
\begin{cor}
Suppose ${\bf u}\not \equiv 0$ solves
\begin{equation*}
\begin{cases}
-\left(|{\bf u}'|^{p-2}{\bf u}'\right)'=\Lambda |{\bf u}|^{p-2}{\bf u} \quad \text{in } (0,1),\\
\hskip1.46truecm{\bf u}(0)=0,\quad {\bf u}(1)=0,
\end{cases}
\end{equation*}
for some $\Lambda>0$. Then $\Lambda=k^p\lambda_p$ for some $k\in \mathbb{N}$ and 
\begin{equation*}
{\bf u}(t)=\frac{{\bf c}}{\Lambda^{1/p}} \omega\left(\Lambda^{1/p}t\right).
\end{equation*}
for some ${\bf c}\in \mathbb{R}^N$, where $\omega$ is the solution of \eqref{littleuode}.
\end{cor}
\begin{proof}
By uniqueness, ${\bf u}$ has the stated form for ${\bf c}={\bf u}'(0)$ and is defined on $(0,1)$.  Since ${\bf u}(1)=0$, 
$$
\omega\left(\Lambda^{1/p}\right)=0
$$
This forces $\Lambda=k^p\lambda_p$.
\end{proof}

 \section{Vectorial Fractional Minimizers}

We now consider 
\begin{equation*}\label{OurFracRatio}
\Lambda^N_{s}=\inf_{{\bf v}}\frac{\displaystyle\int_{\mathbb{R}^n}\int_{\mathbb{R}^n}\frac{|{\bf v}(x)-{\bf v}(y)|^p}{|x-y|^{n+sp}}dxdy}{\displaystyle\int_{\Omega}|{\bf v}|^pdx}.
\end{equation*}
for $N\in \mathbb{N}$, $s\in (0,1)$, and $p\in (1,\infty)$.  In this infimum, ${\bf v}=(v_1,\dots, v_N)\in W^{s,p}_0(\Omega;\mathbb{R}^N)$ is assumed not to be identically $0$. It is a standard exercise to check that the infimum is attained. Using this fact, we will derive analogues of the assertions made above for the ``local" case.  
\begin{lem}\label{FracVecEigenLem}
For each $N\in \mathbb{N}$, 
$$
\Lambda^N_{s}=\Lambda^1_{s}.
$$
\end{lem}
\begin{proof}
Suppose ${\bf u}\in W^{s,p}_0(\Omega;\mathbb{R}^N)$ is a minimizer for $\Lambda^N_{s}$ and set $\omega=|{\bf u}| \in W^{s,p}_0(\Omega)$. By the Cauchy-Schwarz inequality

\begin{equation}\label{CS}\begin{cases}
|\omega(x)-\omega(y)|^2& =|{\bf u}(x)|^2+|{\bf u}(y)|^2-2|{\bf u}(x)||{\bf u}(y)|\\
& \le |{\bf u}(x)|^2+|{\bf u}(y)|^2-2{\bf u}(x)\cdot {\bf u}(y)\\
&= |{\bf u}(x)-{\bf u}(y)|^2.
\end{cases}
\end{equation}
As a result, 
$$
\Lambda^1_{s}\le\frac{\displaystyle\int_{\mathbb{R}^n}\!\!\int_{\mathbb{R}^n}\frac{|\omega(x)-\omega(y)|^p}{|x-y|^{n+sp}}dxdy}{\displaystyle\int_{\Omega}|\omega|^pdx}\le \frac{\displaystyle\int_{\mathbb{R}^n}\!\!\int_{\mathbb{R}^n}\frac{|{\bf u}(x)-{\bf u}(y)|^p}{|x-y|^{n+sp}}dxdy}{\displaystyle\int_{\Omega}|{\bf u}|^pdx}=\Lambda^N_{s}.
$$

\par Alternatively, suppose $u$ is an extremal when $N=1$ and set ${\bf v} ={\bf c} u$ for some ${\bf c}\in \mathbb{R}^N$ with ${\bf c}\neq 0$. 
Then 
$$
\Lambda^N_{s} \le \frac{\displaystyle\int_{\mathbb{R}^n}\!\!\int_{\mathbb{R}^n}\!\!\frac{|{\bf v}(x)-{\bf v}(y)|^p}{|x-y|^{n+sp}}dxdy}{\displaystyle\int_{\Omega}|{\bf v}|^pdx}=
\frac{\displaystyle|{\bf c}|^p\!\!\int_{\mathbb{R}^n}\!\!\int_{\mathbb{R}^n}\!\!\frac{|u(x)-u(y)|^p}{|x-y|^{n+sp}}dxdy}{|{\bf c}|^p\displaystyle\int_{\Omega}|u|^pdx}=\Lambda^1_{s}.
$$
\end{proof}
\begin{thm}\label{fracthm}
Suppose ${\bf u}\in W^{s,p}_0(\Omega;\mathbb{R}^N)$ is a vector minimizer for $\Lambda_s^N$.  Then there is ${\bf c}\in \mathbb{R}^N\setminus\{0\}$ and a scalar minimizer $\omega\in W^{s,p}_0(\Omega;\mathbb{R})$ for  $\Lambda_s^1$ such 
that ${\bf u}={\bf c}\omega$. 
\end{thm}
\begin{proof}
Suppose ${\bf u}=(u_1,\dots,u_N)$.  An inspection of our proof of Lemma \ref{FracVecEigenLem} shows that $|{\bf u}|\in  W^{s,p}_0(\Omega;\mathbb{R})$ is a scalar minimizer and 
$$
{\bf u}(x)\cdot {\bf u}(y)=|{\bf u}(x)|| {\bf u}(y)|
$$
for almost every $x,y\in \mathbb{R}^n$.  By Lagrange's identity \eqref{VecLagrange}, it must be that 
$$
u_i(x)u_j(y)=u_j(x)u_i(y)
$$
for all $i,j$ and almost every $x,y\in \mathbb{R}^n$. Since ${\bf u}$ doesn't vanish identically, there is some $i=1,\dots, N$ and $y\in \mathbb{R}^n$ for which $u_i(y)\neq 0$. Set $\omega=u_i$ and note that 
$$
u_j(x)=\frac{u_j(y)}{u_i(y)}u_i(x)=c_ju(x)
$$
for almost every $x\in \Omega$ and $j=1,\dots, N$.  That is, ${\bf u} ={\bf c} \omega$ with ${\bf c}=(c_1,\dots,c_N)$. As ${\bf u}$ is a vector minimizer, ${\bf c}\neq 0$ and $\omega$ is a scalar minimizer. 
\end{proof}

\subsection*{Acknowledgment}
We thank Manuel Del Pino for an excerpt from \cite{dP} and Erik Lindgren for a fruitful suggestion. Part of this material is based upon work supported by the Swedish Research Council under grant no. 2016-06596 while all three authors were participating in the research program \lq\lq Geometric Aspects of Nonlinear Partial Differential Equations\rq\rq $ $ at Institut Mittag-Leffler in Djursholm, Swe\-den, during the fall of 2022.

\end{document}